\documentclass[12pt]{amsart}
\usepackage{amsmath,amssymb,amsbsy,amsfonts,amsthm,latexsym,
            amsopn,amstext,amsxtra,euscript,amscd}
            \usepackage{url}
\usepackage[colorlinks,linkcolor=blue,anchorcolor=blue,citecolor=blue]{hyperref}
\usepackage{color}

\begin{document}
\bibliographystyle{plain}
\newfont{\teneufm}{eufm10}
\newfont{\seveneufm}{eufm7}
\newfont{\fiveeufm}{eufm5}
%
%
\newfam\eufmfam
              \textfont\eufmfam=\teneufm \scriptfont\eufmfam=\seveneufm
              \scriptscriptfont\eufmfam=\fiveeufm
\def\bbbr{{\rm I\!R}}
\def\bbbm{{\rm I\!M}}
\def\bbbn{{\rm I\!N}}
\def\bbbf{{\rm I\!F}}
\def\bbbh{{\rm I\!H}}
\def\bbbk{{\rm I\!K}}
\def\bbbp{{\rm I\!P}}
\def\bbbone{{\mathchoice {\rm 1\mskip-4mu l} {\rm 1\mskip-4mu l}
{\rm 1\mskip-4.5mu l} {\rm 1\mskip-5mu l}}}
\def\bbbc{{\mathchoice {\setbox0=\hbox{$\displaystyle\rm C$}\hbox{\hbox
to0pt{\kern0.4\wd0\vrule height0.9\ht0\hss}\box0}}
{\setbox0=\hbox{$\textstyle\rm C$}\hbox{\hbox
to0pt{\kern0.4\wd0\vrule height0.9\ht0\hss}\box0}}
{\setbox0=\hbox{$\scriptstyle\rm C$}\hbox{\hbox
to0pt{\kern0.4\wd0\vrule height0.9\ht0\hss}\box0}}
{\setbox0=\hbox{$\scriptscriptstyle\rm C$}\hbox{\hbox
to0pt{\kern0.4\wd0\vrule height0.9\ht0\hss}\box0}}}}
\def\bbbq{{\mathchoice {\setbox0=\hbox{$\displaystyle\rm
Q$}\hbox{\raise
0.15\ht0\hbox to0pt{\kern0.4\wd0\vrule height0.8\ht0\hss}\box0}}
{\setbox0=\hbox{$\textstyle\rm Q$}\hbox{\raise
0.15\ht0\hbox to0pt{\kern0.4\wd0\vrule height0.8\ht0\hss}\box0}}
{\setbox0=\hbox{$\scriptstyle\rm Q$}\hbox{\raise
0.15\ht0\hbox to0pt{\kern0.4\wd0\vrule height0.7\ht0\hss}\box0}}
{\setbox0=\hbox{$\scriptscriptstyle\rm Q$}\hbox{\raise
0.15\ht0\hbox to0pt{\kern0.4\wd0\vrule height0.7\ht0\hss}\box0}}}}
\def\bbbt{{\mathchoice {\setbox0=\hbox{$\displaystyle\rm
T$}\hbox{\hbox to0pt{\kern0.3\wd0\vrule height0.9\ht0\hss}\box0}}
{\setbox0=\hbox{$\textstyle\rm T$}\hbox{\hbox
to0pt{\kern0.3\wd0\vrule height0.9\ht0\hss}\box0}}
{\setbox0=\hbox{$\scriptstyle\rm T$}\hbox{\hbox
to0pt{\kern0.3\wd0\vrule height0.9\ht0\hss}\box0}}
{\setbox0=\hbox{$\scriptscriptstyle\rm T$}\hbox{\hbox
to0pt{\kern0.3\wd0\vrule height0.9\ht0\hss}\box0}}}}
\def\bbbs{{\mathchoice
{\setbox0=\hbox{$\displaystyle     \rm S$}\hbox{\raise0.5\ht0\hbox
to0pt{\kern0.35\wd0\vrule height0.45\ht0\hss}\hbox
to0pt{\kern0.55\wd0\vrule height0.5\ht0\hss}\box0}}
{\setbox0=\hbox{$\textstyle        \rm S$}\hbox{\raise0.5\ht0\hbox
to0pt{\kern0.35\wd0\vrule height0.45\ht0\hss}\hbox
to0pt{\kern0.55\wd0\vrule height0.5\ht0\hss}\box0}}
{\setbox0=\hbox{$\scriptstyle      \rm S$}\hbox{\raise0.5\ht0\hbox
to0pt{\kern0.35\wd0\vrule height0.45\ht0\hss}\raise0.05\ht0\hbox
to0pt{\kern0.5\wd0\vrule height0.45\ht0\hss}\box0}}
{\setbox0=\hbox{$\scriptscriptstyle\rm S$}\hbox{\raise0.5\ht0\hbox
to0pt{\kern0.4\wd0\vrule height0.45\ht0\hss}\raise0.05\ht0\hbox
to0pt{\kern0.55\wd0\vrule height0.45\ht0\hss}\box0}}}}
\def\bbbz{{\mathchoice {\hbox{$\sf\textstyle Z\kern-0.4em Z$}}
{\hbox{$\sf\textstyle Z\kern-0.4em Z$}}
{\hbox{$\sf\scriptstyle Z\kern-0.3em Z$}}
{\hbox{$\sf\scriptscriptstyle Z\kern-0.2em Z$}}}}
\def\ts{\thinspace}

\newtheorem{theorem}{Theorem}
\newtheorem{corollary}[theorem]{Corollary}
\newtheorem{lemma}[theorem]{Lemma}
\newtheorem{claim}[theorem]{Claim}
\newtheorem{cor}[theorem]{Corollary}
\newtheorem{prop}[theorem]{Proposition}
\newtheorem{definition}[theorem]{Definition}
\newtheorem{remark}[theorem]{Remark}
\newtheorem{question}[theorem]{Open Question}


\def\qed{\ifmmode
\squareforqed\else{\unskip\nobreak\hfil
\penalty50\hskip1em\null\nobreak\hfil\squareforqed
\parfillskip=0pt\finalhyphendemerits=0\endgraf}\fi}

\def\squareforqed{\hbox{\rlap{$\sqcap$}$\sqcup$}}

\def \A {{\mathbb A}}
\def \C {{\mathbb C}}
\def \F {{\mathbb F}}
\def \L {{\mathbb L}}
\def \K {{\mathbb K}}
\def \Q {{\mathbb Q}}
\def \Z {{\mathbb Z}}
\def\cA{{\mathcal A}}
\def\cB{{\mathcal B}}
\def\cC{{\mathcal C}}
\def\cD{{\mathcal D}}
\def\cE{{\mathcal E}}
\def\cF{{\mathcal F}}
\def\cG{{\mathcal G}}
\def\cH{{\mathcal H}}
\def\cI{{\mathcal I}}
\def\cJ{{\mathcal J}}
\def\cK{{\mathcal K}}
\def\cL{{\mathcal L}}
\def\cM{{\mathcal M}}
\def\cN{{\mathcal N}}
\def\cO{{\mathcal O}}
\def\cP{{\mathcal P}}
\def\cQ{{\mathcal Q}}
\def\cR{{\mathcal R}}
\def\cS{{\mathcal S}}
\def\cT{{\mathcal T}}
\def\cU{{\mathcal U}}
\def\cV{{\mathcal V}}
\def\cW{{\mathcal W}}
\def\cX{{\mathcal X}}
\def\cY{{\mathcal Y}}
\def\cZ{{\mathcal Z}}
\newcommand{\rmod}[1]{\: \mbox{mod}\: #1}

\def\tcN{\cN^\mathbf{c}}
\def\F{\mathbb F}
\def\Tr{\operatorname{Tr}}
\def\mand{\qquad \mbox{and} \qquad}
\renewcommand{\vec}[1]{\mathbf{#1}}
\def\eqref#1{(\ref{#1})}
\newcommand{\ignore}[1]{}
\hyphenation{re-pub-lished}
\parskip 1.5 mm
\def\lln{{\mathrm Lnln}}
\def\Res{\mathrm{Res}\,}
\def\lcm{\mathrm{lcm}\,}
\def\rad{\mathrm{rad}\,}
\def\F{{\bbbf}}
\def\Fp{\F_p}
\def\fp{\Fp^*}
\def\Fq{\F_q}
\def\ff{\F_2}
\def\ffn{\F_{2^n}}
\def\K{{\bbbk}}
\def \Z{{\bbbz}}
\def \N{{\bbbn}}
\def\Q{{\bbbq}}
\def \R{{\bbbr}}
\def \P{{\bbbp}}
\def\Zm{\Z_m}
\def \Um{{\mathcal U}_m}
\def \Bf{\frak B}
\def\Km{\cK_\mu}
\def\va {{\mathbf a}}
\def \vb {{\mathbf b}}
\def \vc {{\mathbf c}}
\def\vx{{\mathbf x}}
\def \vr {{\mathbf r}}
\def \vv {{\mathbf v}}
\def\vu{{\mathbf u}}
\def \vw{{\mathbf w}}
\def \vz {{\mathbfz}}
\def\\{\cr}
\def\({\left(}
\def\){\right)}
\def\fl#1{\left\lfloor#1\right\rfloor}
\def\rf#1{\left\lceil#1\right\rceil}
\def\AST{\mathcal{A}_{\mathrm{ST}}}
\def\AU{\mathcal{A}_{\mathrm{U}}}

\def \ctE {\widetilde \cE}

\newcommand{\floor}[1]{\lfloor {#1} \rfloor}
\newcommand{\commF}[1]{\marginpar{%
\begin{color}{green}
\vskip-\baselineskip 
\raggedright\footnotesize
\itshape\hrule \smallskip F: #1\par\smallskip\hrule\end{color}}}

\newcommand{\commI}[1]{\marginpar{%
\begin{color}{blue}
\vskip-\baselineskip 
\raggedright\footnotesize
\itshape\hrule \smallskip I: #1\par\smallskip\hrule\end{color}}}

\newcommand{\commM}[1]{\marginpar{%
\begin{color}{magenta}
\vskip-\baselineskip 
\raggedright\footnotesize
\itshape\hrule \smallskip M: #1\par\smallskip\hrule\end{color}}}

\def\rem{{\mathrm{\,rem\,}}}
\def\dist {{\mathrm{\,dist\,}}}
\def\etal{{\it et al.}}
\def\ie{{\it i.e. }}
\def\veps{{\varepsilon}}
\def\eps{{\eta}}
\def\ind#1{{\mathrm {ind}}\,#1}
               \def \MSB{{\mathrm{MSB}}}
\newcommand{\abs}[1]{\left| #1 \right|}

\title[The $r$th moment of the divisor function]{The $r$th moment of the divisor function: an elementary approach}

\author{Florian~Luca}
\address{School of Mathematics, University of the Witwatersrand, Private Bag X3, Wits 2050, South Africa and Department of Mathematics, Faculty of Sciences, University of Ostrava, 30 dubna 22, 701 03 Ostrava 1, Czech Republic}
\email{florian.luca@wits.ac.za}

\author{L\'aszl\'o T\'oth}
\address{Department of Mathematics, University of P\'ecs,
 Ifj\'us\'ag \'utja 6, H-7624, P\'ecs, Hungary}
\email{ltoth@gamma.ttk.pte.hu}


\pagenumbering{arabic}

\maketitle

\centerline{Journal of Integer Sequences {\bf 20} (2017), Article 17.7.4}
\medskip

\begin{abstract}
Let $\tau(n)$ be the number of divisors of $n$. We give an elementary proof of the fact that
$$
\sum_{n\le x} \tau(n)^r =xC_{r} (\log x)^{2^r-1}+O(x(\log x)^{2^r-2}),
$$
for any integer $r\ge 2$. Here,
$$
C_{r}=\frac{1}{(2^r-1)!} \prod_{p\ge 2}\left( \left(1-\frac{1}{p}\right)^{2^r} \left(\sum_{\alpha\ge 0} \frac{(\alpha+1)^r}{p^{\alpha}}\right)\right).
$$
\end{abstract}

\section{Introduction}

Let $\tau(n)$ be the number of divisors of $n$. Ramanujan \cite{Ram} stated without proof that, given any real number $\varepsilon>0$, the estimate
$$
\sum_{n\le x} \tau(n)^2=x(A(\log x)^3+B  (\log x)^2+C\log x+D)+O(x^{3/5+\varepsilon})
$$
holds with $A=\pi^{-2}$. An elementary proof of the asymptotic formula
$$
\sum_{n\le x} \tau(n)^2\sim Ax(\log x)^3,
$$
as $x\to\infty$, appears in several places (see, for example, \cite[Thm.~7.8]{Nat}). Wilson \cite{Wil} proved Ramanujan's claim and generalized it by showing that for any integer $r\ge 2$ one has
$$
\sum_{n\le x} \tau(n)^r=x(C_{r,1} (\log x)^{2^r-1}+C_{r,2}(\log x)^{2^r-2}+\cdots+C_{r,2^r})+O(x^{\frac{2^r-1}{2^r+2}+\varepsilon}).
$$
Note that when $r=2$, Wilson's error term is better than the one claimed by Ramanujan. We are not aware even of elementary proofs for the asymptotic formula
$$
\sum_{n\le x} \tau(n)^r \sim C_{r}  x(\log x)^{2^r-1}
$$
as $x\to\infty$ for any $r\ge 2$. In this note, we give an elementary proof of the following more general result.

\begin{theorem}
\label{thm:1}
Let $k$ be a positive integer and $f(n)$ be a multiplicative function which on  prime powers $p^{\alpha}$ satisfies
$$
f(p)=k\quad {\text{\rm and}}\quad f(p^{\alpha})=\alpha^{O(1)}\quad {\text{for~all~primes}}~p~{\text{and~integers}}~ \alpha\ge 2,
$$
where the constant implied by the above $O$ is uniform in $p$.
Then
$$
\sum_{n\le x} f(n)=x C_{f} (\log x)^{k-1}+O(x(\log x)^{k-2})
$$
where
$$
C_{f}=\frac{1}{(k-1)!} \left(\prod_{p\ge 2} \left(1-\frac{1}{p}\right)^k \left(\sum_{\alpha\ge 0} \frac{f(p^{\alpha})}{p^{\alpha}}\right)\right).
$$
\end{theorem}

In the case $f(n)=\tau(n)^r$ for integer $r\ge 1$, Theorem \ref{thm:1} applies with $k=2^r$.

The only facts that we use are Abel's summation formula, the M\"obius inversion formula, the elementary estimate
\begin{equation}
\label{eq:sum1overn}
\sum_{n\le t} \frac{1}{n}=\log t+\gamma+O(1/t)
\end{equation}
valid for all real $t\ge 1$,  and the fact that the counting function of the {\it squarefull} numbers $s\le t$ is $O(t^{1/2})$, where $s$ is squarefull if and only if $p^2\mid s$ for all prime factors $p$ of $s$, all provable by elementary means.

\section{A lemma}

\begin{lemma}
\label{lem:1}
Assume that $r$ is a positive integer and $f(n)$ is some arithmetic function such that
\begin{equation}
\label{eq:xtominus1/2}
\sum_{n\le x} f(n)=\sum_{j=0}^r c_j (\log x)^j+O(x^{-1/2+o(1)}),
\end{equation}
for some constants $c_j$, $j=0,\ldots,r$. Then
\begin{equation}
\label{eq:logxovern}
\sum_{n\le x} f(n)(\log(x/n))^k=\sum_{\ell=0}^{k+r} C_{\ell }(\log x)^{\ell}+O(x^{-1/2+o(1)}),
\end{equation}
holds for all positive integers $k$ with some constants $C_0,\ldots,C_{k+r}$. Here, if $\ell\in \{k,k+1,\ldots,k+r\}$, then
\begin{equation}
\label{eq:Ckj}
C_{\ell}:=c_{\ell-k} \left(1+(\ell-k)\sum_{i=1}^k \frac{(-1)^i}{\ell-k+i}\binom{k}{i}\right).
\end{equation}
Furthermore, if $r\ge t\ge 1$ are positive integers and
\begin{equation}
\label{eq:errorlog}
\sum_{n\le x} f(n)=\sum_{j=t}^r c_j (\log x)^j+O((\log x)^{t-1}),
\end{equation}
then
\begin{equation}
\label{eq:logxovern2}
\sum_{n\le x} f(n) (\log(x/n))^k=\sum_{j=k+t}^{k+r} C_j (\log x)^j+O((\log x)^{t+k-1}).
\end{equation}

\end{lemma}

\begin{proof}
We show how to deduce \eqref{eq:logxovern} out of \eqref{eq:xtominus1/2} with the leading coefficients given by \eqref{eq:Ckj}. Let
$$
A(x)=\sum_{n\le x} f(n).
$$
Then
$$
A(x)=\sum_{j=0}^r c_j(\log x)^j+R(x),
$$
where $|R(x)|=x^{-1/2+o(1)}$ as $x\to\infty$.
Let $i\ge 1$. Put
$$
B_i(x):=\sum_{n\le x} f(n)(\log n)^i.
$$
Then, by the Abel summation formula and by interchanging the order between the summation and the integration, we get
\begin{eqnarray*}
\label{eq:Bj}
B_i(x) & = & A(x)(\log x)^i-i\int_1^x A(t) \left(\frac{(\log t)^{i-1}}{t} \right)dt\nonumber\\
& = & \sum_{j=0}^r \left( c_j(\log x)^{j+i}-i\int_1^x \left(\frac{c_j(\log t)^{j+i-1}}{t}\right) dt \right) \\
& - & i\int_1^x \frac{(\log t)^{i-1} R(t)}{t} dt +  R(x)(\log x)^i\\
& = & \sum_{j=0}^r \left(c_j(\log x)^{j+i}-\frac{c_j i}{j+i} (\log t)^{j+i}\Big|_1^x\right)+\\
& - & i\int_1^{\infty} \frac{(\log t)^{i-1} R(t)}{t} dt+i\int_x^{\infty} \frac{(\log t)^{i-1} R(t)}{t} dt +R(x) (\log x)^i \\
& = & \sum_{j=0}^r \frac{c_j j}{j+i} (\log x)^{j+i}+D_i+O(x^{-1/2+o(1)}),\nonumber\\
\end{eqnarray*}
where
$$
D_i:=-i\int_1^{\infty} \frac{(\log t)^{i-1} R(t)}{t} dt
$$
In the above, we used the fact that $|R(t)|\le t^{-1/2+o(1)}$ as $t\to\infty$ to deduce that the above integral converges and that its tail from $x$ to infinity as well as the other errors are $O(x^{-1/2+o(1)})$ as $x\to\infty$.
Using the binomial formula and the above arguments, we have
\begin{eqnarray*}
C_k(x) & := & \sum_{n\le x} f(n)(\log (x/n))^k \\
& = & \sum_{i=0}^k (-1)^i \binom{k}{i} (\log x)^{k-i} \sum_{n\le x} f(n) (\log n)^i\\
& = & \sum_{n\le x} f(n)+ \sum_{i=1}^k (-1)^{i} \binom{k}{i} (\log x)^{k-i}B_i(x)\\
& = & \sum_{\ell=0}^{k+r} C_{\ell} (\log x)^{\ell}+O(x^{-1/2+o(1)}),
\end{eqnarray*}
where $C_{\ell}$ are given by formula \eqref{eq:Ckj} for $\ell\ge k$. For $\ell=1,\ldots,k-1$, the coefficient $C_{\ell}$ involves the expression $D_{\ell}$. The deduction of \eqref{eq:logxovern2} out of \eqref{eq:errorlog} is immediate by similar arguments.
\end{proof}

\section{The proof of Theorem~\ref{thm:1}}
 Let $f_0(n):=f(n)$. Recursively define $f_j(n )$ such that
 $$
 f_{j-1}(n)=\sum_{d\mid n} f_j(d),\quad j=1,2,\ldots.
 $$
 By M\"obius inversion,
 $$
 f_j(n)=\sum_{d\mid n} \mu(d) f_{j-1}(n/d).
 $$
 On primes
 $$
 f_j(p)=f_{j-1}(p)-1,\quad j=1,2,\ldots.
 $$
 Since $f_0(p)=k$, we get that $f_j(p)=k-j$. In particular, $f_k(p)=0$. Further, for $\alpha\ge 2$, we have that
 $$
 f_j(p^{\alpha})=f_{j-1}(p^{\alpha})-f_{j-1}(p^{\alpha-1}).
 $$
 Since $f_0(p^{\alpha})=\alpha^{O(1)}$ it follows that $f_j(p^{\alpha})=\alpha^{O(1)}$ for all $j\ge 2$. The constant in $O(1)$ might depend on $j$. Further,
 $$
 \sum_{\alpha\ge 0} \frac{f_j(p^{\alpha})}{p^{\alpha}}=\left(1-\frac{1}{p}\right)\sum_{\alpha\ge 0} \frac{f_{j-1}(p^{\alpha})}{p^{\alpha}},\quad j=1,2,\ldots,
 $$
 therefore
 $$
 \sum_{\alpha\ge 0} \frac{f_j(p^{\alpha})}{p^{\alpha}}=\left(1-\frac{1}{p}\right)^j \sum_{\alpha\ge 0} \frac{f(p^{\alpha})}{p^{\alpha}},\quad j=0,1,\ldots
 $$
 Put
 $$
 E_j:=\prod_{p\ge 2} \left(\sum_{\alpha\ge 0} \frac{f_j(p^{\alpha})}{p^{\alpha}}\right)=\prod_{p\ge 2} \left(\left(1-\frac{1}{p}\right)^j \sum_{\alpha\ge 0} \frac{f(p^{\alpha})}{p^{\alpha}}\right).
 $$
Fix $j\ge 1$. Then
 $$
F_{j-1}(x):= \sum_{n\le x} \frac{f_{j-1}(n)}{n}=\sum_{n\le x} \frac{1}{n}\sum_{d\mid n} f_{j}(d)=\sum_{d\le x} f_{j}(d)\sum_{\substack{n\le x\\ d\mid n}} \frac{1}{n}.
 $$
 In the inner sum, we write an $n\le x$ which is a multiple of $d$ as $n=dm$ for some integer $m\le x$. We get
 \begin{eqnarray}
 \label{eq:j-1}
F_{j-1}(x) & = & \sum_{d\le x} \frac{f_j(d)}{d} \sum_{m\le x/d} \frac{1}{m}=\sum_{d\le x} \frac{f_j(d)}{d}\left(\log(x/d)+\gamma+O(d/x)\right)\nonumber\\
& = & \sum_{d\le x} \frac{f_j(d)}{d} \log(x/d)+\gamma\sum_{d\le x} \frac{f_j(d)}{d}+O\left(\frac{1}{x} \sum_{d\le x} |f_j(d)|\right)
\end{eqnarray}
for $j=1,2,\ldots$. When $j=k$, since $f_k(p)=0$, it follows that $f_k(d)=0$ if $d$ is not squarefull. Thus, when $j=k$ in the right--hand side of \eqref{eq:j-1}, we have
$$
\sum_{d\le x} \frac{f_k(d)}{d} \log(x/d)+\gamma \sum_{d\le x} \frac{f_k(d)}{d}+O\left(\frac{1}{x} \sum_{d\le x} |f_k(d)|\right).
$$
Note that
\begin{eqnarray}
\label{eq:E}
\sum_{d\le x} \frac{f_k(d)}{d} & = & \sum_{d\ge 1} \frac{f_k(d)}{d} +O\left(\sum_{d>x} \frac{|f_k(d)|}{d}\right)=E_k+O\left(\sum_{\substack{d\ge x\\ d~{\text{\rm squarefull}}}} \frac{1}{d^{1+o(1)}}\right)\nonumber\\
& = & E_k+O(x^{-1/2+o(1)}),
\end{eqnarray}
where for the error term we used the fact that $|f_k(d)|=|\tau(d)|^{O(1)}=d^{o(1)}$ as $d\to\infty$ and the Abel summation formula to conclude that
$$
\sum_{\substack{d>x\\ d~{\text{\rm squarefull}}}} \frac{1}{d^{1+o(1)}}\le x^{-1/2+o(1)}\quad {\text{\rm as}}\quad x\to\infty.
$$
Further, we have
\begin{eqnarray}
\label{eq:F}
\sum_{d\le x} \frac{f_k(d)}{d}(-\log d+\gamma) & = & \sum_{d\ge 1} \frac{f_k(d)(-\log d+\gamma)}{d}+O\left(\sum_{\substack{d>x\\ d~{\text{\rm squarefull}}}} \frac{|f_k(d)|\log d}{d}\right)\nonumber\\
& : = & F_k+O(x^{-1/2+o(1)})
\end{eqnarray}
as $x\to\infty$, by a similar argument since $|f_k(d)|\log d\le d^{o(1)}$ as $d\to \infty$. Finally
\begin{equation}
\label{eq:G}
\sum_{d\le x} |f_k(d)|\le x^{1/2+o(1)},
\end{equation}
again since $f_k(d)=0$ if $d$ is not squarefull. Collecting \eqref{eq:E}, \eqref{eq:F} and \eqref{eq:G} and putting them into \eqref{eq:j-1} with $j=k$, we get
$$
F_{k-1}(x)=\sum_{n\le x} \frac{f_{k-1}(n)}{n}=E_k \log x+F_k+O(x^{-1/2+o(1)}).
$$
In a similar way,
$$
G_{k-1}(x):=\sum_{n\le x} \frac{|f_{k-1}(n)|}{n}=E_k' \log x+F_k'+O(x^{-1/2+o(1)}).
$$
for some (maybe different) constants $E_k'$ and $F_k'$.  We now apply Lemma \ref{lem:1} in order to find recursively $F_{k-2}(x), F_{k-3}(x),~\ldots,~F_0(x)$. We claim, by induction on $j$, that
\begin{equation}
\label{eq:j}
F_{k-j}(x)=A_j (\log x)^{j}+B_j(\log x)^{j-1}+O((\log x)^{j-2})
\end{equation}
for $j=2,\ldots,k$. At $j=1$, this is so with $A_1=E_k$, $B_1=F_k$ and the error term is better, namely $O(x^{-1/2+o(1)})$. In order to realize the induction step from $j=1$ to $j=2$, we use
the first part of Lemma 1 with $r=1$, whereas for the induction step from $j\ge 2$ to $j+1$ we use the second part of Lemma \ref{lem:1} with $r=j$ and $t=j-1$. Assuming that \eqref{eq:j} holds for $j\ge 1$, we have, by \eqref{eq:j-1},
\begin{eqnarray*}
F_{k-j-1}(x) & = & \sum_{d\le x} \frac{f_{k-j-1}(d)}{d}=\sum_{d\le x} \frac{f_{k-j}(d)}{d} \log(x/d)+\gamma\sum_{d\le x} \frac{f_{k-j}(d)}{d}\\
& + & O\left(\frac{1}{x} \sum_{d\le x} |f_{k-j}(d)|\right).
\end{eqnarray*}
By Lemma \ref{lem:1}, we get that the right hand side is
\begin{eqnarray*}
&& \frac{A_j}{j+1} (\log x)^{j+1}+\left(\frac{B_j}{j}+\gamma A_{j}\right) (\log x)^{j}\\ & + & O\left((\log x)^{j-1}+\frac{1}{x} \sum_{d\le x} |f_{k-j}(d)|\right)\\
&& :=A_{j+1}(\log x)^{j+1}+B_{j+1} (\log x)^{j}+O\left((\log x)^{j-1}+\frac{1}{x} \sum_{d\le x} |f_{k-j}(d)|\right),
\end{eqnarray*}
where
$$
A_{j+1}=\frac{A_j}{j+1},\quad {\text{\rm and}}\quad B_{j+1}=\gamma A_j+\frac{B_j}{j}.
$$
Thus, we note that $A_j=E_k/j!$. It remains to deal with the sum in the error term. But the exact same approach applies to $|f_{k-j}(n)|$. That is $g_0(n)=|f_{k-j}(n)|$ satisfies the same conditions as our initial $f_0(n)$ with $k$ replaced by $k-j$. Thus,
$$
\sum_{d\le x} \frac{|f_{k-j}(d)|}{d}=C_j(\log x)^{j} +D_j (\log x)^{j-1}+O((\log x)^{j-2}),
$$
where for $j=1$, the error term is $O(x^{-1/2+o(1)})$ as $x\to\infty$.
By Abel summation, we get that
\begin{eqnarray*}
\sum_{d\le x} |f_{k-j}(d)| & = & x(C_j(\log x)^j+D_j (\log x)^{j-1}+O((\log x)^{j-2}))\\
& - & \int_1^x (C_j(\log t)^j+D_j (\log t)^{j-1}+O((\log t)^{j-2})) dt\\
& = & O(x(\log x)^{j-1}),
\end{eqnarray*}
which is sufficient for us. This completes the induction procedure and shows that at $j=k$ we have
$$
\sum_{n\le x} \frac{f(n)}{n} =\frac{1}{k!} E_k (\log x)^{k}+B_k (\log x)^{k-1}+O((\log x)^{k-2}).
$$
Abel summation formula once again gives
\begin{eqnarray*}
\sum_{n\le x} f(n) & = & \left(\frac{E_k}{k!} (\log x)^{k}+B_k(\log x)^{k-1}+O((\log x)^{k-2})\right) x\\
& - & \int_1^x \left(\frac{E_k}{k!} (\log t)^{k}+B_k(\log t)^{k-1}+O((\log t)^{k-2})\right) dt\\
& = &
\frac{E_k}{(k-1)!} x(\log x)^{k-1}+O(x(\log x)^{k-2}),
\end{eqnarray*}
which is what we wanted.

\section{Acknowledgments}

We thank the referee for a careful reading of the manuscript. This work
was done when both authors visited the Max Planck Institute of
Mathematics in Bonn, Germany in February 2017. They thank the
Institution for the invitation and support.

\end{document}